\documentclass{amsart}
\usepackage{amsmath,amsthm,amsfonts,amscd,amssymb}

\newtheorem{thm}{Theorem}[section]
\newtheorem{theorem}{Theorem}[section]

 \newtheorem{lemma}[thm]{Lemma}
 \newtheorem{prop}[thm]{Proposition}

\def \l {\left}
\def \r {\right}

\def \l {\left}
\def \r {\right}

\def \bx {{\bf x}}

\def\proof{\noindent \textsc{Proof:}\ }
\def\endproof{$\Box$ \medskip}

\def\Fqsq{\mathbb{F}_{q^2}}
\def\calK{\mathcal{K}}

\def\calP{\mathcal{P}}
\def\calO{\mathcal{O}}

\def\FF{\mathbb{F}}

\def\RR{\mathbb{R}}
\def\real{\mathbb{R}}
\def\ZZ{\mathbb{Z}}
\def\zed{\mathbb{Z}}

\def\Aut{\operatorname{Aut}}
\def\GL{\operatorname{GL}}

\def\SL2{\mathrm{SL}_{2}}

\def\proof{\noindent \textsc{Proof:}\ }
\def\endproof{ \hfill \parbox{0.5cm}{$\Box$}}

\newcommand{\rmv}[1]{}

\def \spn {\mathrm{span}}

\def \Div {\mathrm{Div}}

\newcommand{\nH}{Q}
\newcommand{\si}{\sigma}
\newcommand{\1}{{\bf 1}}
\newcommand{\2}{{\bf 2}}
\newcommand{\3}{{\bf 3}}
\newcommand{\4}{{\bf 4}}
\newcommand{\5}{{\bf 5}}
\newcommand{\6}{{\bf 6}}
\newcommand{\ta}{\{\sigma_1\}}

\begin{document}

\title{Lattices from Hermitian function fields}
\author[A. B\"ottcher]{Albrecht B\"ottcher}
\author[L. Fukshansky]{Lenny Fukshansky}
\author[S. R. Garcia]{\\ Stephan Ramon Garcia}
\author[H. Maharaj]{Hiren Maharaj}\thanks{Fukshansky acknowledges support by Simons Foundation grant \#279155, Garcia acknowledges support by NSF grant DMS-1265973}
\rmv{
\address{Fakult\"at f\"ur Mathematik, TU Chemnitz, 09107 Chemnitz, Germany,}
\email{aboettch@mathematik.tu-chemnitz.de}
\address{Department of Mathematics, Claremont McKenna College, 850 Columbia Ave, Claremont, CA 91711, USA, {\em E-mail address: {\tt lenny@cmc.edu}}}
\email{lenny@cmc.edu}
\address{Department of Mathematics, Pomona College, 610 N. College Ave, Claremont, CA 91711, USA, {\em E-mail address: {\tt stephan.garcia@pomona.edu}}}
\email{stephan.garcia@pomona.edu}
\address{Department of Mathematics, Claremont McKenna College, 850 Columbia Ave, Claremont, CA 91711, USA, {\em E-mail address: {\tt hmahara@g.clemson.edu}}}
\email{hmahara@g.clemson.edu}
}


\begin{abstract} We consider the well-known Rosenbloom-Tsfasman function field lattices
in the special case of Hermitian function fields. We show that in this case
the resulting lattices are generated by their minimal vectors, provide
an estimate on the total number of minimal vectors, and derive properties
of the automorphism groups of these lattices. Our study continues previous
investigations of lattices coming from elliptic curves and finite Abelian groups.
The lattices we are faced with here are more subtle than those considered previously, and
the proofs of the main results require the replacement of the existing linear algebra approaches by
deep results of Gerhard Hiss on the factorization of functions with particular divisor support into lines and their inverses.
\end{abstract}

\maketitle

\section{Introduction}
\label{intro}

A lattice is a discrete subgroup in a Euclidean space $\real^n$. Lattice theory aims to understand geometric properties of lattices and to use them for a variety of applications, such as discrete optimization problems or coding theory. Some of the geometrically most interesting lattices, in particular those possessing many symmetries, come from several well-studied algebraic constructions.
These include, for instance, ideal lattice constructions from number fields and polynomial rings; see, e.g.,~\cite{bayer1}, \cite{bayer2} and~\cite{lub_mic}, respectively, for a detailed overview of these constructions.
A series of prominent algebraic constructions of lattices are also presented in~\cite{tv}.
In this paper, we focus our attention on an important algebraic construction, originally introduced by Rosenbloom and Tsfasman in~\cite{rt} and later described in~\cite{tv}, that of {\em function field lattices}.

The construction of function field lattices given in \cite{tv} is as follows. Let $F$ be an algebraic function field (of a single variable) with the finite field $\FF_q$ as its full field of constants, where $q$ is a prime power. Let $\calP = \{ P_0,P_1,P_2,\ldots, P_{n-1}\}$ be the set of rational
places of $F$.  For each place $P_i$, let $v_i$ denote the corresponding normalized discrete valuation and let  $\calO_\calP^*$ be the set of all nonzero functions  $f\in F$ whose divisor has support contained in the set $\calP$. Then $\calO_\calP^*$ is an Abelian group, $\sum_{i=0}^{n-1} v_i(f) = 0$ for each $f \in \calO_\calP^*$, and we  define
$$\deg f := \sum_{v_i(f) > 0} v_i(f) = \frac{1}{2} \sum_{i=0}^{n-1} |v_i(f)|.$$
{ Let
$\varphi_\calP : \calO_\calP^* \to \ZZ^n$ be the group homomorphism given by $$\varphi_\calP(f) = (v_0(f),v_1(f), \ldots, v_{n-1}(f)).$$
Then $L_\calP := \mathrm{Image}(\varphi_\calP)$ is a finite-index sublattice of the root lattice
$$A_{n-1} = \left\{ \bx =(x_0,\ldots,x_{n-1}) \in \zed^n : \sum_{i=0}^{n-1} x_i = 0 \right\}$$
with minimum distance
\begin{equation} \label{mind}
d(L_\calP) \ge \min \left\{  \sqrt{2\deg f} : f \in \calO_\calP^*\setminus\FF_q \right\},
\end{equation}
and
\begin{equation} \label{det}
\det L_\calP \le \sqrt{n} h_F \le \sqrt{n} \left ( 1 + q  + \frac{n-q-1}{g} \right )^g,
\end{equation}
where $g$ is the genus of $F$ and $h_F$ is the divisor class number of $F$, that is, the size of the  group of divisor classes of $F$ of  degree 0, denoted by $\mathrm{Cl}^0(F)$. Here, as in \cite{tv}, we can identify $\ZZ^n$ with the set of all divisors with support in $\calP$ and $A_{n-1}$ with the set of all divisors of degree~$0$. We will often make use of this identification when working with lattice vectors by working with the corresponding divisors instead.

Unless stated otherwise, we will use notation as
in \cite{stich}. We write $F /\mathbb F_q$ to mean that $F$ is a
global function field with full field of constants $\mathbb{F}_q$. Let
$g=g(F)$ denote the genus of $F$.  If $P$ is a rational place of
$F$, that is, a place of degree one, then $v_P$ denotes the
discrete valuation corresponding to $P$.
We write  $\mathrm{supp} \, A$ for the support of a
divisor $A$. The divisor of a
function $f \in F \setminus \{ 0 \}$ is denoted by $(f)$ and the divisor class of a place $P$
by $[P]$.

We will study  lattices from Hermitian function fields.
The Hermitian
function field $\mathbb{F}_{q^2}(x,y)/\mathbb{F}_{q^2}$ has defining equation
$y^q+y=x^{q+1}$.
The purpose of this paper is to show that the lattices which arise from Hermitian function fields
are generated by minimal vectors and are hence well-rounded. Recall that a lattice $L$ of rank $k$ is called well-rounded if it contains $k$ linearly independent minimal vectors, i.e., vectors of Euclidean norm equal to $d(L)$, and that $L$ is said to be generated by minimal vectors if the set of all minimal vectors of $L$ spans $L$ over $\zed$. The statement that $L$ is generated by minimal vectors is equivalent to the statement that $L$ is well-rounded for $k \leq 4$ and is strictly stronger for $k \geq 5$; in other words, there exist well-rounded lattices of rank 5 and greater whose minimal vectors generate a sublattice of index greater than 1. See~\cite{martinet} for further information.

In \cite{bo_et_al}, we studied sublattices $L_G$ of the root lattice $A_{N-1}$
which are of the form
\begin{equation}
L_G=\left\{\bx  =(x_0,\ldots,x_{N-1}) \in A_{N-1}: \sum_{j=1}^{N-1} x_j g_j=0\right\},\label{alb1}
\end{equation}
where $G=\{g_0 := 0,g_1,\ldots,g_{N-1}\}$ is a finite (additively written) Abelian group.
We showed that $\det L_G=N^{3/2}$ and that except for $G=\zed_4$, the lattice $L_G$ always has a basis of minimal vectors and
is hence well-rounded. Here and in what follows, $\zed_m:=\zed/m\zed$.
Such lattices emerge in particular when applying the above construction to elliptic curves over $\mathbb{F}_q$. The groups $G$ coming from elliptic curves were characterized by R\"uck~\cite{Ruck},
and for these groups, the results of ~\cite{bo_et_al} had previously been established by  Min Sha~\cite{sha}.

The only elliptic curve among the Hermitian curves $y^q+y=x^{q+1}$ over $\mathbb{F}_{q^2}$ is the curve
$y^2+y=x^3$ over $\mathbb{F}_4$, which corresponds to $q=2$; see~\cite{fu_ma}. In that case $G$ is isomorphic to $\zed_3^2$, that is, we have $N=9$ and $\det L_G=27$.

Except for the case $q=2$, the Hermitian function fields considered here lead to a class of lattices which
are more general than the lattices~(\ref{alb1}). Namely,
given a finite Abelian group $G$ and a subset $S = \{ g_0 := 0, g_1, \ldots, g_{n-1} \}$
of $G$, we define the sublattice $L_G(S)$ of $A_{n-1}$ by
\[L_G(S) =\left\{\bx=(x_0,\ldots,x_{n-1}) \in A_{n-1}: \sum_{j=1}^{n-1} x_j g_j=0\right\}.\]
It turns out that unless $S = G$, in which case $L_G(G)=L_G$, the situation changes dramatically: there are many $S$
for which $L_G(S)$ is well-rounded and there are many $S$ for which $L_G(S)$ is not
well-rounded. In general it is a delicate problem to decide which of the two possibilities occurs in a concrete case.

As the result of the abstract construction of function field lattices outlined above, we obtain $L_{\calP}=L_G(S)$,
where $S$ is the set $S = \{ [P_i-P_0]: 0\le i \le n-1 \}$ of divisor classes
and $G$ is the subgroup of the divisor class group
$\mathrm{Cl}^0(F)$ generated by $S$. Thus, in this case $S$ is not simply a subset of $G$, but a generating set for $G$.
If the function field is specified to be $\mathbb{F}_{q^2}(x,y)/\mathbb{F}_{q^2}$ with the defining equation
$y^q+y=x^{q+1}$, the group $G$ can be shown to be isomorphic to $\zed_{q+1}^{q^2-q}$, which is just the above $\zed_3^2$ for $q=2$,
and $S$ becomes a set of $q^3+1$ generators of $G$. For $q=2$, $S$ has 9 elements and therefore
coincides with $G$. However, if $q > 2$, then the set $S$ is much smaller than $G$.
In the light of what was said at the end of the preceding paragraph, it is therefore a rather surprising fact
that the lattices $L_{\calP}$ coming from the curves $y^q+y=x^{q+1}$ over $\mathbb{F}_{q^2}$ are always well-rounded
and even more, are generated by their minimal vectors.

To strengthen the surprise and to emphasize the subtlety of the matter we mention the following.
The Klein
curve $K$ is defined by  $$(x + y + 1)^4 + (xy + x + y)^2 + xy(x + y + 1)  =0.$$
Over $\mathbb{F}_4$, the set $\calP$ of $\mathbb{F}_4$-rational points of $K$ contains 14 elements,
and the curve yields a rank 13 lattice $L_\calP$ of dimension 14.
A quick computation using Magma~\cite{magma} shows that
the lattice $L_\calP$ has $168$ minimal vectors. These minimal vectors
generate a sublattice of $L_\calP$ of index $2$. This implies that $L_\calP$
is well-rounded but not generated by the minimal vectors.

This paper is organized as follows. In Section~\ref{herm} we set the basic notation of Hermitian function fields. In Section~\ref{divisors} we give a detailed characterization of divisors coming from lines in a Hermitian function field. We derive formulas for the minimal distance and determinant of lattices coming from Hermitian function fields via the above construction in Sections~\ref{min_dist} and~\ref{det_sect}, respectively. In Section~\ref{gen_min} we establish our main result, which asserts that these lattices are generated by their minimal vectors. Our proof makes essential use of the results of Hiss~\cite{gh}. We do not know of a proof along the linear algebra approaches developed in~\cite{bo_et_al} and~\cite{sha}. In Section~\ref{automorphisms}
we investigate properties of automorphism groups of these lattices, as well as more general lattices coming from generating sets in Abelian groups. Finally, in Section~\ref{number_min_vecs} we establish a lower bound on the number of minimal vectors of lattices from Hermitian function fields, which is the same as the kissing number of these lattices.

\section{Hermitian  function fields: basic facts}
\label{herm}

The following are some basic facts about these function fields. Let
$H$ denote $\mathbb{F}_{q^2}(x,y)$  with the defining equation
$y^q+y=x^{q+1}$ over $\mathbb{F}_{q^2}$. Thus, we use $H$ for the function field $F$ in the  Rosenbloom-Tsfasman construction outlined
above.
The genus of $H$ is $g=\frac{q(q-1)}{2}$ and $H$ has $n = q^3+1$ places of degree $1$ over $\Fqsq$, namely
\begin{itemize}
\item the common pole $Q_\infty$ of $x$ and $y$, and \item for
each $\alpha \in \Fqsq$, there are $q$ elements $\beta \in \Fqsq$
such that $\beta^q+\beta=\alpha^{q+1}$, and for
 each such pair
$(\alpha,\beta)$ there is a unique place $P_{\alpha,\beta}$ of $H$
of degree one with $x(P_{\alpha,\beta})=\alpha$ and
$y(P_{\alpha,\beta})=\beta$.
\end{itemize}
 Let
$$
\mathcal{K} :=\{ (\alpha, \beta) \in \mathbb{F}_{q^2}^2 :
\beta^q+\beta=\alpha^{q+1}\}.
$$
We let $\calP$ stand for the set of rational places of $H$:
 the place $Q_\infty$  and the places} $P_{\alpha,\beta}$  indexed by $(\alpha,\beta) \in \calK$.
For
each $(\alpha,\beta)\in \mathcal{K}$, set
$$\tau_{\alpha,\beta}:= y-\beta -
\alpha^q(x-\alpha). $$
Observe that $\tau_{\alpha,\beta} = y - \alpha^qx + \beta^q$ and note
that $\tau_{\alpha,\beta}$ is the tangent line to the Hermitian curve
at the point $(\alpha,\beta)$.
 If one views $H$ as a
Kummer extension over $\mathbb{F}_{q^2}(y)$, the rational places of
$\mathbb{F}_{q^2}(y)$ behave as follows:
\begin{itemize}
\item For each $\gamma \in \mathbb{F}_{q^2}$ such that $\gamma^q + \gamma =
0$, the place $y-\gamma$ is totally ramified.  If $\gamma^q +
\gamma \ne 0$, the place $y-\gamma$ splits completely in $H$.
 \item The pole of $y$ is totally ramified.
\end{itemize}
 We remark  that
\begin{equation}\label{taux}
\tau_{\alpha,\beta}^q+\tau_{\alpha,\beta}=(x-\alpha)^{q+1}.
\end{equation}
 We therefore have $H= \mathbb{F}_{q^2}(x,y) = \mathbb{F}_{q^2}(\tau_{\alpha,\beta}, x)$, and so we may
view $H$ as a Kummer extension of $\Fqsq(\tau_{\alpha,\beta})$.   It follows that
 the divisor of
$\tau_{\alpha,\beta}$ is
$$
\left( \tau_{\alpha,\beta} \right) =(q+1)P_{\alpha,\beta}-(q+1)Q_\infty.
$$
Following the usual convention for rational function fields, we
denote the places of $\mathbb{F}_{q^2}(\tau_{\alpha,\beta})$ by their corresponding monic
irreducible polynomials, except in the case of the  place at
infinity,  which we denote by $P_\infty(\tau_{\alpha,\beta})$. For any $\gamma
\in \mathbb F_{q^2}$ satisfying $\gamma^q + \gamma = 0$, we have
$\tau_{\alpha,\beta} - \gamma = \tau_{\alpha, \beta+\gamma}$.  Thus, we will
write ``the place $\tau_{\alpha, \beta+\gamma}$ in
$\mathbb{F}_{q^2}(\tau_{\alpha,\beta})$'' to mean the place $\tau_{\alpha,\beta} - \gamma $.

\section{Divisors of lines}
\label{divisors}

We will call functions of the form $ax+by+c$ ($a,b,c\in \Fqsq$ with not both $a,b$ zero) lines.
Furthermore, by points on the line $ax+by+c$ we mean the points of intersection of the line
$ax + by + c$ with the Hermitian curve $y^q + y = x^{q+1}$.
In the next result we determine the divisor of every line and thus   obtain
the points (of $\calK$) which lie on a line.

\begin{lemma} \label{tauram}
Let $H/\mathbb{F}_{q^2}$ denote a Hermitian function field and let
$\gamma\in\mathbb{F}_{q^2}$.

{\rm (a)}
If $\gamma^q + \gamma = 0$, the place $\tau_{\alpha,\beta} - \gamma =
\tau_{\alpha, \beta+\gamma}$ in $\mathbb{F}_{q^2}(\tau_{\alpha,\beta})$ is totally
ramified in the extension $H/\mathbb{F}_{q^2}(\tau_{\alpha,\beta})$. The divisor of
$\tau_{\alpha,\beta} - \gamma$ is
$$(\tau_{\alpha,\beta} - \gamma) = (q+1)P_{\alpha,\beta+\gamma} - (q+1)Q_\infty.$$
The line $\tau_{\alpha,\beta} - \gamma$ is a tangent line.

{\rm (b)} The pole $P_\infty(\tau_{\alpha,\beta})$ of $\tau_{\alpha,\beta}$ is totally ramified
in the extension $H/\mathbb{F}_{q^2}(\tau_{\alpha,\beta})$.

{\rm (c)} If $\gamma^q + \gamma \ne 0$, the place $\tau_{\alpha,\beta} - \gamma $
in $\mathbb{F}_{q^2}(\tau_{\alpha,\beta})$ splits completely in the extension
$H/\mathbb{F}_{q^2}(\tau_{\alpha,\beta})$ and the divisor of
$\tau_{\alpha,\beta} - \gamma$ is
\begin{equation}
 \sum_{i=0}^q  P_{\alpha+\delta \zeta^i,\beta + \gamma + \alpha^q \delta \zeta^i}
-(q+1)Q_\infty
\end{equation}
where $\zeta$ is a primitive $(q+1)$st root of unity in $\mathbb{F}_{q^2}$ and
$\delta\in \mathbb{F}_{q^2}^*$ is  such that
$\gamma^q+\gamma = \delta^{q+1}$.
The points of $\calK$ which lie on the line $\tau_{\alpha,\beta} - \gamma$ are precisely
$$(\alpha+\delta \zeta^i,\beta + \gamma + \alpha^q \delta \zeta^i)\;\;(0\le i \le q).$$
The line $\tau_{\alpha,\beta} - \gamma$ is not a tangent line.

{\rm (d)} Suppose that $f = y+bx + c$. Let $\delta\in\Fqsq$ be such that
$\delta^{q+1} = b^{q+1} -(c^q + c)$. Then the points of $\calK$ which lie
on the line $f$ are precisely
$$(-b^q +\delta \zeta^i,b^{q+1}-c -b \delta \zeta^i) \; \; (0\le i \le q).$$
It follows that $f$ is a tangent line if and only if $\delta=0$ if and only if $(-b^q,c^q) \in \calK$ {\rm (}if and only if
$(-b,c) \in \calK\,${\rm )}. If $f$ is a tangent line then
$f= \tau_{-b^q,c^q}$. If $\delta \ne 0$, then $f$ contains exactly $q+1$ points from~$\calK$.

{\rm (e)} Suppose that $f= x-c$. Then the divisor of $f$ is
\begin{equation}
 \sum_{d} P_{c,d} - qQ_\infty
\end{equation}
where the sum is over the $q$ solutions $d \in \mathbb{F}_{q^2}$ to $d^q+d= c^{q+1}$.
\end{lemma}

\begin{proof}
Parts $(a), (b)$, and $(e)$ follow from viewing $H$ as a Kummer extension of
$\Fqsq(\tau_{\alpha,\beta})$.

Proof of $(c)$: Since the trace map from
$\mathbb{F}_{q^2}$ to $\mathbb{F}_q$ is given by  $ z\mapsto z^q + z$ and 
the norm map $z\mapsto z^{q+1}$ from $\mathbb{F}_{q^2}^*$ to $\mathbb{F}_q^*$  is surjective,
there exists a $\delta\in \mathbb{F}_{q^2}^*$ such that
$\gamma^q+\gamma = \delta^{q+1}$. Let $\zeta$ be a primitive $(q+1)$st root of unity in
$\mathbb{F}_{q^2}$.  Then the set of all solutions $x$   to $\gamma^q + \gamma = (x-\alpha)^q$
  is given by  $x-\alpha = \delta \zeta^i$, that is, $x= \alpha+ \delta \zeta^i$  ($0\le i \le q$).
Now $\tau_{\alpha,\beta}  = y- \beta - \alpha^q(x-\alpha) =  y - \beta - \alpha^q \delta \zeta^i = \gamma$, and so
$y = \beta + \gamma + \alpha^q \delta \zeta^i$. This proves $(c)$.

Proof of $(d)$:
Let $\alpha = -b^q$. Then  $b = -\alpha^q$.
 If $\beta \in \mathbb{F}_{q^2}$  satisfies
$\beta^q + \beta = \alpha^{q+1} = b^{q+1}$, then
$f= y - \alpha^qx+c  =\tau_{\alpha,\beta} -\gamma$ where $\gamma = \beta^{q} - c$.
  Now observe that $\gamma^q + \gamma = b^{q+1} -(c^q + c) = \delta^{q+1}$.
By $(c)$, it follows that the points on the line $f$ are
$(\alpha+\delta \zeta^i,\beta + \gamma + \alpha^q \delta \zeta^i) = (-b^q +\delta \zeta^i,b^{q+1}-c -b \delta \zeta^i)$, as required.
\end{proof}

\section{The minimum distance of the lattice}
\label{min_dist}

\begin{theorem}\label{mindist}
The minimum distance of the lattice is $\sqrt{2q}$ and the minimum degree of every non-constant function in $\calO_k^*$
is $q$.
\end{theorem}
\begin{proof}
Choose a point $P = (\alpha,\beta)$ on the affine Hermitian curve  and choose two distinct  lines
 $L_1$, $L_2$ passing through     $P$ which are not `vertical', that is, neither of the lines are of the form
$x-a$. Such a pair of lines is easily constructed. Indeed, choose two distinct nonzero slopes $M_1$ and $M_2$.
Then, by the surjectivity of the Frobenius endomorphism, there exist $m_1,m_2\in \Fqsq$ such that $M_1 = m_1^q$ and $M_2 = m_2^q$.
Find constants $c_1, c_2$ such that the lines $L_1 = y-m_1^qx+c_1$ and
$L_2 = y-m_2^q x + c_2$ both pass through the point $P$. One easily sees that the lines are
of the form $L_i = \tau_{m_i,n_i}-\gamma_i$ for some $m_i, n_i,\gamma_i\in \Fqsq$    ($i=1,2$).  From
 the previous section we know that the lattice vector corresponding to  $L_1/L_2$  has $q$ ones, $q$
 minus ones, and that the remaining entries are all zero. The $\ell_1$-norm of the corresponding lattice vector is $2q$
 and the Euclidean
 norm of    that
lattice vector  equals  $\sqrt{2q}$. Thus, the minimum of the $\ell_1$-norms of the nonzero lattice  
vectors is at most $2q$.

Now choose a function $f$ corresponding to a nonzero lattice vector.
Note that the sum of the positive entries of the corresponding lattice vector equals   minus the
sum of the negative entries,  which also equals the degree $[H:\mathbb{F}_{q^2}(f)]$.
Now $H$ has $q^3+1$ rational places, so
$q^3 + 1 \le [H:\mathbb{F}_{q^2}(f)](q^2 + 1)$,  whence
$[H:\mathbb{F}_{q^2}(f)] \ge (q^3 + 1 )/(q^2+1)$ and thus  $[H:\mathbb{F}_{q^2}(f)] \ge q$.
 Consequently, the  minimum $\ell_1$-norm is at least $2q$.
From above it follows that the minimum $\ell_1$ norm is exactly $2q$.

 Now let $f$ be a function which corresponds to a nonzero lattice vector of minimum $\ell_2$-norm.
 Then $\|f\|_2^2  \ge 2 \|f\|_1 \ge 2q$ with equality throughout if $f = L_1/L_2$. It follows
 that the minimum norm of the lattice is $\sqrt{2q}$.
\end{proof}

\section{The exact determinant of the lattice}
\label{det_sect}

First we recall  part of the proof of Tsfasman and Vladut for  the lower bound (\ref{mind}) on the minimum distance of a lattice
from a function field
  given in \cite{tv}. Let $F$ be a function field over a finite field $\FF_q$ and
let $\calP$ be a nonempty set of rational places of $F$.
The set $O_\calP^*$ is the set of all nonzero functions $f$ whose support is contained in $\calP$.
Put $n= \#\calP$.  We use the obvious one-to-one correspondence between the set of divisors with support contained in $\calP$ (we denote this set by $\Div_\calP^0$) and the root lattice $A_{n-1}$.
The set of all divisors of functions from $O_\calP^*$ is a sublattice of $A_{n-1}$ denoted by $L_\calP$.
Now
$A_{n-1}/L$ is isomorphic to $\Div_0(\calP)/\mathrm{Princ}(\calP)$ where
$\mathrm{Princ}(\calP)$ is the set of all principal divisors with support in $\calP$.
Thus  $[A_{n-1}:L] = |\Div_0(\calP)/\mathrm{Princ}(\calP)|$ and it follows that
$$\mathrm{Vol}(\RR^n/L) = \mathrm{Vol}(\RR^n/A_{n-1}) [A_{n-1}:L] = \sqrt{n}  |\Div_0(\calP)/\mathrm{Princ}(\calP)|.$$
The group $\Div_0(\calP)/\mathrm{Princ}(\calP)$ is isomorphic to a subgroup of the divisor class group, and hence  the volume is bounded above by $\sqrt{n} h_F$ where $h_F$ is the class number of~$F$.

\begin{prop}
The determinant of the lattice is
$\sqrt{q^3+1} \cdot (q+1)^{q^2-q}$.
\end{prop}

\begin{proof}
In the case of Hermitian function fields, the divisor classes $P-Q$ modulo $\mathrm{Princ}(\calP)$, where $P,Q\in \calP$, generate the group $\Div_0(\calP)/\mathrm{Princ}(\calP)$ (see \cite{hess}). Thus the group  $\Div_0(\calP)/\mathrm{Princ}(\calP)$ is isomorphic to the divisor class group of the Hermitian function field. The class group of the Hermitian function field  is
 isomorphic to $\ZZ_{q+1}^{q^2-q}$, and so the class number is $(q+1)^{q^2-q}$. Since $n = q^3 + 1$, the  result  follows
 from the discussion above.
\end{proof}

\section{The lattice is generated by its minimal vectors}
\label{gen_min}

From Lemma \ref{tauram} and Theorem \ref{mindist} we infer the following lemma.

\begin{lemma}\label{minv}
If  $L_1$ and $L_2$ are two distinct lines then
$(L_1/L_2)$  {\rm (}or $(L_2/L_1)${\rm )} is a minimal vector if one of the following holds:

$\bullet$ $L_1$ and $L_2$ are of the form $x-\alpha$,

$\bullet$ one of $L_1, L_2$  is of the form $x-\alpha$ and the other is a non-tangent line {\rm (}of the form $y + ax + c${\rm )}
and both lines have exactly one point of intersection,

$\bullet$  both $L_1$ and $L_2$ are non-tangent lines {\rm (}of the form $y+ax+c${\rm )} with a common point of intersection.
\end{lemma}

G. Hiss \cite{gh} showed that every function in $\calO_\calP^*$  is the product of functions of the form $ax+bx+c$ and their inverses.  This fact is essential in the proof of the next  result.
\begin{thm}
The lattice $L_\calP$ is generated by the minimal vectors and  is hence  well-rounded.
\end{thm}

 \begin{proof}
Since the lattice is generated by the divisors of the lines \cite{gh}, it suffices
to show that every such divisor is an integer linear combination of minimal vectors
of the lattice. We   call a line $L = ax+by+c$  {\em good}  if  the divisor of
$L$ is an integer linear combination of minimal vectors. Thus our goal is to show
that all lines are good. We consider different cases. Throughout the proof,
$\zeta\in \Fqsq$ denotes a primitive $(q+1)$st root of unity.

{\em Case 1:   Suppose that $d,e\in \Fqsq$ satisfy $d^q + d =e^{q+1}$ with $e\ne 0$. We show that
the lines $y-d$ and $x-e$ are good.}
Let $d_1 = d,d_2, \ldots, d_q$ be all the solutions to $y^q + y = e^{q+1}$.
Then
$$\prod_{i=1}^q  ( y - d_i ) = y^q + y - e^{q+1} = x^{q+1} - e^{q+1} =
\prod_{i=0}^q  (x-\zeta^i e).$$
It follows that
$$x-e  = \prod_{i=1}^q \left( \frac{y-d_i }{x-\zeta^i e} \right).$$
The lines $y-d_i$ and $x-e\zeta^i$ have just one point of intersection and
$y-d_i$ is a non-tangent line (since $d_i$ has nonzero trace).
So by Lemma \ref{minv} it follows that the lattice vector corresponding to
the function $\frac{y-d_i}{x-\zeta^i e}$ is a minimal vector. Since the divisor
of $x-e$ is the sum of the divisors of the functions $\frac{y-d_i}{x-\zeta^i e}$, $1\le i \le q$,   we arrive at the conclusion
that the line $x-e$ is good.

On the other hand,  we also have
$$y-d  = (x-e)(x-e\zeta)\prod_{i=2}^q \l( \frac{x-e\zeta^i}{y-d_i} \r).$$
Since each factor on the right corresponds to  a lattice vector which is either a minimal vector or
which  can be expressed as a linear combination of minimal vectors, it follows that
the line $y-d$ is good.

{\em Case 2: We show that every non-tangent line of the form
 $L = y+bx +c$ is good.}
Since $L$ is a non-tangent line,   we infer from Lemma \ref{tauram} that
$(-b,c) \not \in \calK$, that is, $c^q + c \ne (-b)^{q+1} = b^{q+1}$.

Let $\alpha = -b^q$, so that  $b= -\alpha^q$. Note that $\alpha^{q+1} = b^{q+1}$
 and let $\beta\in \Fqsq$ be any solution to
$\beta^q + \beta = \alpha^{q+1} \; ( = b^{q+1})$.
Then $L= y-\alpha^q x + \beta^q +c-\beta^q = \tau_{\alpha,\beta} - d$ where
$d= \beta^q-c$.  Observe that $d^q + d = \beta^q + \beta -(c^q + c) = b^{q+1}-(c^q +c) \ne 0$.

Choose $e \in \Fqsq$ such that $d^q + d = e^{q+1}$ (so $c^q + c = b^{q+1} - e^{q+1}$).
Note that $e \ne 0$. Let
$d_1 = d,d_2, \ldots, d_q \in \Fqsq$ be all the solutions to $y^q + y = e^{q+1}$.
Writing $\tau$ for $\tau_{\alpha,\beta}$ we get that
\begin{equation}\label{maintaueqn}
\prod_{i=1}^q  (\tau - d_i )                   = \tau^q + \tau - e^{q+1} = (x-\alpha)^{q+1} - e^{q+1} =
\prod_{i=0}^q  (x-\alpha - \zeta^i e).
\end{equation}
 It follows that
 \begin{equation}\label{newx}
 x - \alpha - e = \prod_{i=1}^q \left(  \frac{\tau-d_i}{x-\alpha- \zeta^i e} \right).
 \end{equation}
 Since $d_i^q + d_i = e^{q+1} = d^q + d  \ne 0$, we obtain from Lemma \ref{tauram}$(c)$
that the lines $\tau-d_i$ are not tangent lines.
The line $\tau - d_i$ intersects the line $x-\alpha - \zeta^i e$ at
exactly one point:  $(\alpha + \zeta^i e ,  \beta + d_i + e \alpha^q \zeta^i)$. Moreover,
this point belongs to $\calK$   because
 \begin{eqnarray*}
& & (\beta + d_i + e \alpha^q \zeta^i)^q + (\beta + d_i + e \alpha^q \zeta^i)\\
& &  = \beta^q + \beta + d_i^q + d_i + e^q \alpha \zeta^{iq}  + e \alpha^q + \zeta^i\\
& &  = \alpha^{q+1} + e^{q+1}  + e^q \alpha \zeta^{iq}  + e \alpha^q + \zeta^i
 = (\alpha + \zeta^i e )^{q+1}.
\end{eqnarray*}
 Thus the lattice vectors corresponding to functions $\frac{\tau-d_i}{x-\alpha- \zeta^i e}$
$(1 \le i \le q)$ are minimal vectors.  It follows from
equation (\ref{newx}) that the line $x-\alpha- e$ is good.
Replacing $e$ by $\zeta e$ in the above argument, we see that $x-\alpha - \zeta e$ is also
good.
From equation (\ref{maintaueqn}) we get
\begin{equation}\label{taud}
L = \tau - d  = (x-\alpha-e)(x-\alpha - e\zeta)\prod_{i=2}^q \l( \frac{x-\alpha - \zeta^i e}{\tau-d_i} \r).
\end{equation}
Since each factor on the right corresponds to  a lattice vector which is either a minimal vector or
which  can be expressed as a linear combination of minimal vectors, we conclude that
the line $L$ is good.

Note that Case 1 is actually implied as a special case of the proof of Case 2 with $b=0$.

{\em Case 3:  We prove that the tangent line $\tau_{0,0} = y$ is good.}
First of all observe that
$$y^{q+1}-x^{q+1} =  y^{q+1} -y^q - y = (y-1)^{q+1} - 1 = \prod_{i=0}^q (y-1-\zeta^i).$$
On the other hand we also have that  $y^{q+1}-x^{q+1} = \prod_{i=0}^q ( y - \zeta^i x)$, and so
$$\prod_{i=0}^q (y-1-\zeta^i) = \prod_{i=0}^q ( y - \zeta^i x).$$
Since $-1$ is a $(q+1)$st root of unity, there is a unique index $j  \in \{0,\ldots,q\}$ such that
 $\zeta^j = -1$ (actually $j=0$ in characteristic $2$ and $j=(q+1)/2$ in odd characteristics). Then
\begin{equation}\label{eqny}
y = (y-\zeta^j x) \prod_{i=0,\, i\ne j}^q \left( \frac{y-\zeta^i x}{y-1-\zeta^i} \right).
\end{equation}
The points on the line $y-(1+\zeta^i)$ are  $((1+\zeta^i)\zeta^k,1+\zeta^i) \in \calK$  with
$k=0,1,\ldots, q$.
This implies that the line $y-1-\zeta^i$  (for $i \ne j$) intersects the line $y-\zeta^i x$ in exactly one point,
namely  $( (1+\zeta^i)\zeta^{q+1-i}, 1+\zeta^i)$. This point belongs to
$\calK$ because
\begin{eqnarray*}
\l( (1+\zeta^i)\zeta^{q+1-i} \r)^{q+1} & = & (1 + \zeta^i)^{q+1} = (1 + \zeta^{iq})(1+\zeta^i) \\
& = & 1 + \zeta^{iq} + \zeta^i + 1 = (1 +\zeta^i)^q  + (1 + \zeta^i).
\end{eqnarray*}
Note that the lines $y - \zeta^i x$ $(1\le  i \le q)$ are not
tangent lines  since  $(\zeta^i)^{q+1} = 1  \ne 0$ and thus $(-\zeta^i, 0) \not \in \calK$.

It follows that the lattice vector corresponding to the functions
$\frac{y-\zeta^i x}{y-1-\zeta^i}$  $(0\le i \le q, \; i \ne j)$ is a minimal vector.  As the line $y-\zeta^jx$ is not a tangent line, it
is good by Case 2.  It   now results from (\ref{eqny}) that the line $y$ is good.

 {\em Case 4: For every $(\alpha,\beta) \in \calK$, the tangent line  $\tau_{\alpha,\beta} = y -\alpha^q x + \beta^q$ is good.}
Set $\tau := \tau_{\alpha,\beta}$ and $x_\alpha = x-\alpha$.  Note that $(-\alpha, \beta^q) \in \calK$.
By \cite[page 238]{stich}, there exists a $\sigma \in \mathrm{Aut}(H/\Fqsq)$
such that $\sigma(x) = x-\alpha$ and $\sigma(y) = y - \alpha^q x +\beta^q = \tau$.
Observe that, in the notation of \cite{stich}, we are using $(d,e) = (-\alpha, \beta^q)$. Applying $\sigma$ to equation (\ref{eqny}) we get
\begin{equation}\label{eqntau}
 \tau = (\tau -\zeta^j x_\alpha ) \prod_{i=0,i \ne j}^q \left( \frac{\tau-\zeta^i x_\alpha }{\tau-1-\zeta^i} \right).
\end{equation}
Note that one could also derive this identity in the same way as in Case 3.
By \cite[Lemma 3.5.2]{stich}, a place  $Q$ is a common zero of $\sigma(y-1-\zeta^i)$ and
$\sigma(y-\zeta^i x)$ if and only if  $\sigma^{-1}(Q)$ is a common zero of
$y-1-\zeta^i$ and $y-\zeta^ix$.
Thus, using the results from Case 3, we see that  the line $\tau-1-\zeta^i = \sigma(y-1-\zeta^i)$ intersects the line
$\tau-\zeta^i x_\alpha = \sigma(y-\zeta^i x)$ at exactly one point. Moreover, again by \cite[Lemma 3.5.2]{stich},
the lines $\tau-\zeta^i x_\alpha = \sigma(y-\zeta^i x)$  and $\tau-1-\zeta^i = \sigma(y-1-\zeta^i)$
are non-tangent lines, both of the form $y+ax+c$. Thus by Lemma \ref{minv}, the lattice vectors corresponding to the
functions  $\frac{\tau-\zeta^i x_\alpha }{\tau-1-\zeta^i}$ $(0\le i \le q, \; i \ne j)$ are all minimal. Since the line
$\tau -\zeta^j x_\alpha $ is good, it follows from equation (\ref{eqntau}) that the line $\tau$ is good as well.

{\em Case 5:  We finally show that the line $x$ is good.}
We start with the observation that
$$y^q + y - (x^q + x) = x^{q+1} - x^q -x   = (x-1)^{q+1} - 1 = \prod_{i=0}^q (x-1 - \zeta^i).$$
On the other hand,
$$y^q + y - (x^q + x) =(y-x)^q +(y-x)  = \prod_{i = 1}^q (y-x- \rho_i),$$
where $\rho_1,\ldots, \rho_q\in \Fqsq$ are all the solutions to $\rho^q + \rho = 0$. Thus
\begin{equation}\label{xeqn}
\prod_{i=0}^q x-1 - \zeta^i   =  \prod_{i = 1}^q (y-x- \rho_i). \end{equation}
Let $z_1,z_2, \ldots, z_q$ be a renumbering of $1+\zeta^0,1+\zeta^1, \ldots, 1+\zeta^{j-1}, 1+\zeta^{j+1},
\ldots, 1+\zeta^{q}$ (recall that $\zeta^j = -1$).
Then it follows from equation (\ref{xeqn}) that
\begin{equation}\label{eqnx} x = \prod_{i=1}^q \left( \frac{y-x - \rho_i}{x-z_i} \right).
\end{equation}
Observe that the two lines
$x -( 1+\zeta^m )$ and
$y - x - \rho_i$ intersect at the point $(1+ \zeta^m, 1+\zeta^m + \rho_i)$. Moreover
the point  $(1+ \zeta^m, 1+\zeta^m + \rho_i)$ belongs to $\calK$  since
$$ (1+\zeta^m + \rho_i)^q + (1+\zeta^m + \rho_i)
= \rho_i^i + \rho_i^q +    1 + \zeta^{mq} + 1 + \zeta^m
= (1 + \zeta^m)^{q+1}.$$
The line $y-x-\rho_i$ is a non-tangent line   because  $(1,-\rho_i) \not \in \calK$.  Thus,  by
Lemma~\ref{minv},  the lattice vector corresponding to
the function $\frac{y-x - \rho_i}{x-z_i}$ is a minimal vector. It therefore follows from equation (\ref{eqnx}) that
the line $x$ is good.
\end{proof}

\section{Automorphism groups of lattices}
\label{automorphisms}

In this section we discuss automorphisms of lattices coming from generating sets in Abelian groups and specifically address the case of
Hermitian and other curves.

\subsection{Lattices from Abelian groups}

Let
$G = \{ g_0, g_1, \dots, g_{n-1}, \dots, g_{N-1} \}$
be an Abelian group with $g_0 = 0$, and let
$S = \{ g_0, g_1, \dots, g_{n-1} \}$
be a subset of $G$. Put
$$L_G = \left\{ \left( x_1,\dots,x_{N-1}, -\sum_{i=1}^{N-1} x_i \right) : x_1,\dots,x_{N-1} \in \zed,\ \sum_{i=1}^{N-1} x_i g_i = 0 \right\} \subseteq A_{N-1}$$
and
$$L_G(S) = \left\{ \left( x_1,\dots,x_{n-1}, -\sum_{i=1}^{n-1} x_i \right) : x_1,\dots,x_{n-1} \in \zed,\ \sum_{i=1}^{n-1} x_i g_i = 0 \right\} \subseteq A_{n-1}.$$
Hence $L_G$ and $L_G(S)$ are full-rank sublattices of the root lattices $A_{N-1}$ and $A_{n-1}$, respectively. We denote the
vectors in $L_G(S)$ by $X = \left( x, -\sum_{i=1}^{n-1} x_i \right)$ with $x = (x_1,\dots,x_{n-1})$ in $\zed^{n-1}$.
The automorphism group $\Aut(L_G(S))$ is defined as the group of all maps of $L_G(S)$ onto itself which extend to linear isometries of $\spn_{\real} A_{n-1}$.
It is easily seen that a map $\tau \in \Aut(L_G(S))$
is necessarily of the form
$$\tau(X)=\tau \left(x_1, \ldots, x_{n-1}, -\sum_{i=1}^{n-1} x_i\right) = \left(Ux,-\sum_{i=1}^{n-1} (Ux)_i \right)$$
with some matrix $U \in \GL_{n-1}(\zed)$.  We therefore identify $\Aut(L_G(S))$ with a subgroup of $\GL_{n-1}(\zed)$.
Moreover, we identify the symmetric group $S_{n-1}$ with the group of all permutation matrices in $\GL_{n-1}(\zed)$.
For the analogous notation regarding the lattice $L_G$, we refer to~\cite{bo_et_al}.

If $S$ is a subgroup of $G$ and $\Aut(S)$ denotes for the automorphism group of $S$, then
$\Aut(L_G(S)) \cap S_{n-1} \cong \Aut(S)$
by Theorem~1.4 of \cite{bo_et_al}.
More generally, if $S$ is any subset of $G$, let us define
$$\Aut(G,S) := \left\{ \sigma \in \Aut(G) : \sigma(g_i) \in S\ \forall g_i \in S \right\}.$$
Notice that every element of $\Aut(G)$ fixes $0$ and permutes $g_1,\dots,g_{N-1}$, which allows us to identify $\Aut(G)$ with a subgroup of $S_{N-1}$, the group of permutations on $N-1$ letters. Think of $S_{n-1}$ as the subgroup of $S_{N-1}$ consisting of all permutations of the first $n-1$ letters. Each element of $\Aut(G,S)$ induces a permutation of $S$, and hence gives rise to an element of $S_{n-1}$. Let us write $\Aut(G,S)^*$ for the group of permutations of $S$ which are extendable to automorphisms of $G$. In other words, every element of $\Aut(G,S)^*$ is a restriction $\sigma|_S : S \to S$ of some element $\sigma \in \Aut(G,S)$ and every element of $\Aut(G,S)$ arises as an extension $\hat{\tau} : G \to G$ of some element $\tau \in \Aut(G,S)^*$.

\begin{thm} \label{auto1}
With notation as above, $\Aut(G,S)^*$ is isomorphic to a subgroup of $\Aut(L_G(S)) \cap S_{n-1}$.
If $S$ is a generating set for $G$, then
$$\Aut(G,S)^* \cong \Aut(L_G(S)) \cap S_{n-1}.$$
\end{thm}

\proof
First notice that every element of $\Aut(G,S)^*$ fixes $0$ and permutes the elements $g_1,\dots,g_{n-1}$. Hence $\Aut(G,S)^*$ can be identified
with a subgroup of the symmetric group~$S_{n-1}$. We denote this subgroup by $\nH$.
Our first objective is to construct an injective group homomorphism $\Phi: \nH \to \Aut(L_G(S)) \cap S_{n-1}$.

Let $\sigma \in \nH$. Then, for every $g_i \in S$, $\sigma(g_i) = g_{\sigma(i)}$ and $\sigma(0)=0$. If
$$X = \left(x_1,\dots,x_{n-1}, -\sum_{i=1}^{n-1} x_i \right) \in L_G(S),$$
then $\sum_{i=1}^{n-1} x_i g_i = 0$. Notice that $\sigma^{-1}$ is also in $\nH$, and so
$$0 = \sigma^{-1}(0) = \sum_{i=1}^{n-1} x_i g_{\sigma^{-1}(i)} = \sum_{i=1}^{n-1} x_{\sigma(i)} g_i.$$
Now define $\tau=\Phi(\sigma)$ on $L_G(S)$ by
$$\tau \left(x_1,\dots,x_{n-1}, -\sum_{i=1}^{n-1} x_i \right) := \left(x_{\sigma(1)},\dots,x_{\sigma(n-1)}, -\sum_{i=1}^{n-1} x_{\sigma(i)} \right).$$
It is clear that $\tau$ maps $L_G(S)$ onto itself.
The matrix $U \in \GL_{n-1}(\zed)$ corresponding to $\tau$ is obviously a permutation matrix.
Consequently, $\tau$ is in $\Aut(L_G(S)) \cap S_{n-1}$. Finally, it is readily seen that $\Phi$ is an injective group homomorphism. Hence $\Phi(\nH) \leq \Aut(L_G(S)) \cap S_{n-1}$.

Now assume that $S$ is a generating set for $G$. We will show that $\Phi(\nH) = \Aut(L_G(S)) \cap S_{n-1}$. Indeed, let $\tau \in \Aut(L_G(S)) \cap S_{n-1}$. If
$$X = \left( x_1,\dots,x_{n-1},-\sum_{i=1}^{n-1} x_i \right) \in L_G(S),$$
then $\tau(X) = (x_{\sigma(1)},\dots,x_{\sigma(n-1)},-\sum_{i=1}^{n-1} x_{\sigma(i)})$ with some $\sigma \in S_{n-1}$, and since
both $X$ and $\tau(X)$ belong to  $L_G(S)$, it follows that
$0 = \sum_{i=1}^{n-1} x_i g_i = \sum_{i=1}^{n-1} x_{\sigma(i)} g_i.$
We have $\tau =\Phi(\sigma)$ with $\sigma: S \to S$ defined by
$\sigma(g_i) := g_{\sigma(i)}$ and $\sigma(0) := 0$. 

To complete the proof, we only need to show that $\sigma$ extends to an automorphism of $G$.
For this, notice that every element $g \in G$ can be written as $g = \sum_{i=1}^{n-1} a_i g_i$ with $a_1,\dots,a_{n-1} \in \zed$, since $S$ generates $G$. Then define
$$\sigma(g) := \sum_{i=1}^{n-1} a_i \sigma(g_i) = \sum_{i=1}^{n-1} a_i g_{\sigma(i)}.$$
To check that this is well-defined, suppose that
$\sum_{i=1}^{n-1} a_i g_i = \sum_{i=1}^{n-1} b_i g_i$
for some integers $a_1,\dots,a_{n-1}, b_1,\dots,b_{n-1}$, and hence $\sum_{i=1}^{n-1} (a_i-b_i) g_i = 0$. Then
$$Y := \left( a_1-b_1,\dots,a_{n-1}-b_{n-1}, \sum_{i=1}^{n-1} (b_i-a_i) \right) \in L_G(S),$$
and so
$$\tau^{-1}(Y) = \left( a_{\sigma^{-1}(1)}-b_{\sigma^{-1}(1)},\dots,a_{\sigma^{-1}(n-1)}-b_{\sigma^{-1}(n-1)}, \sum_{i=1}^{n-1} (b_{\sigma^{-1}(i)}-a_{\sigma^{-1}(i)}) \right)$$
is  in $L_G(S)$, meaning that
$0 = \sum_{i=1}^{n-1} (b_{\sigma^{-1}(i)}-a_{\sigma^{-1}(i)}) g_i = \sum_{i=1}^{n-1} (b_i-a_i) g_{\sigma(i)}.$
Hence $\sum_{i=1}^{n-1} a_i g_{\sigma(i)} = \sum_{i=1}^{n-1} b_i g_{\sigma(i)}$, and so $\sigma$ is well-defined.

Our definition readily implies that $\sigma$ is a group homomorphism. To see that it is surjective, suppose that $g \in G$. Then
$g=\sum_{i=1}^{n-1} a_i g_i$ for some $a_1,\dots,a_{n-1} \in \zed$. For each $1 \leq i \leq n-1$,
$\sigma^{-1}(i) \in \{ 1,\dots,n-1\}$ and $\sigma^{-1}(i) \neq \sigma^{-1}(j)$ whenever $1 \leq i \neq j \leq n-1$,
since $\sigma, \sigma^{-1} \in S_{n-1}$ are bijections. Then let $h = \sum_{i=1}^{n-1} a_i g_{\sigma^{-1}(i)}$,
and notice that $\sigma(h) = g$, hence $\sigma : G \to G$ is a surjective group homomorphism.
Since $G$ is finite, infectivity of $\sigma$ is implied, thus $\sigma \in \Aut(G)$, and so $\tau \in \Phi(H)$. This completes the proof.
\endproof

\subsection{An example}

Let $G=\{0,1,2,3,4,5,6\}=\zed_7 \:(:=\zed/7\zed)$. Then every subset $S \subseteq G$ containing $0$ and at least one other element is a generating
set of $G$. Let, for instance, $S=\{0,1,2,4\}$, which, in the above notation, is an example
of $S=\{0,g_1,\ldots,g_{n-1}\}$ with $n-1=3$. The lattice $L_G(S)$ is
\[\{(x_1,x_2,x_3,-(x_1+x_2+x_3))\in \zed^4: x_1+2x_2+4x_3 = 0 \;\mbox{mod}\; 7\}\subseteq A_3.\]
It can be checked directly that the minimal distance $d(L_G(S))$ equals $\sqrt{6}$ and that $L_G(S)$ has exactly
$6$ minimal vectors, the three vectors $(-2,1,0,1)$, $(0,-2,1,1)$, $(1,0,-2,1)$ and their negatives. As the first three
of these vectors are linearly independent, it follows that $L_G(S)$ is well-rounded. For $j=1,\ldots,6$, we denote by $\si_j \in
\Aut(G)$ the automorphism which sends $1$ to $j$ and thus $k$ to $kj$ modulo 7. Clearly, $\Aut(G)=\{\si_1,\ldots,\si_6\}$. The automorphisms $\si_i$ which
leave $S$ invariant are just $\si_1,\si_2, \si_4$. Consequently, $\Aut(G,S)^*=\{\si_1,\si_2,\si_4\}$ and Theorem~\ref{auto1}
tells us that $\Aut(L_G(S)) \cap S_3 \cong \{\si_1,\si_2,\si_4\}$.

Table 1 below reveals what happens if $S$ ranges over all possible proper subsets of $G=\zed_7$.
The column of the table headed by $n-1=k$  shows the numbers $g_1\ldots g_k$
for the $\tbinom{6}{k}$ possible sets $S=\{0,g_1,\ldots,g_k\}$.
The lattice $L_G(S)$ is well-rounded if and only if the corresponding
numbers  $g_1\ldots g_k$ are in boldface. We also indicated the minimal distance of
$L_G(S)$. For example, the first $8$ lattices in the column $n-1=3$
have the minimal distance $\sqrt{6}$ and the remaining $12$  lattices
in the column $n-1=3$ have the minimal distance $2$. Also added is the group
$\Aut(G,S)^*$ in each case.

Altogether we have $62=2^6-2$ lattices. Exactly $26$ of them are well-rounded
and the remaining $36$ lattices are not well-rounded. It is not a surprise that the group
$\Aut(G,S)^*$ is nontrivial if the lattice is well-rounded, but it is
surprising that this group may also be nontrivial for lattices which are not well-rounded.
Of course, it would be nice to have the implication ``$|{\rm Aut}(G,S)^*| > 1 \;\Rightarrow\;$
the lattice is well-rounded'', but the table shows that this is not true.

\begin{figure}[h]
\[\begin{array}{|c|c|c|c|c|}
\hline
n-1=1 & n-1=2 & n-1=3 & n-1=4 & n-1=5 \\
\hline
d=\sqrt{98} & d=\sqrt{14} & d=\sqrt{6} & d=2 & d=2\\
\1 \ta & \1\3 \ta & \1\2\4 \{\si_1,\si_2,\si_4\}& 1234 \ta & \1\2\3\4\5\ta \\
\2 \ta & \1\5 \ta & \1\2\5 \ta & 1235 \ta & \1\2\3\4\6 \ta \\
\3 \ta & \2\3 \ta  & \1\3\6 \ta  & 1236 \ta & \1\2\3\5\6\ta \\
\4 \ta & \2\6 \ta & \1\4\6 \ta & 1245 \ta & \1\2\4\5\6\ta \\
\5 \ta & \4\5 \ta & \2\3\4 \ta & 1246 \ta & \1\3\4\5\6\ta \\
\6 \ta & \4\6 \ta & \2\5\6 \ta & 1256 \{\si_1,\si_6\}& \2\3\4\5\6\ta \\
& d=\sqrt{6} & \3\4\5 \ta & 1345 \ta &  \\
& 12 \ta & \3\5\6 \{\si_1,\si_2,\si_4\} & 1346 \{\si_1,\si_6\} &  \\
& 14 \ta & d=2 & 1356 \ta & \\
& 16 \{\si_1,\si_6\} & 123\ta  & 1456 \ta & \\
& 24 \ta & 126 \ta & 2345 \{\si_1,\si_6\}& \\
& 25 \{\si_1,\si_6\} & 134 \ta & 2346 \ta & \\
& 34 \{\si_1,\si_6\} & 135 \ta & 2356 \ta & \\
& 35 \ta & 145 \ta & 2456 \ta & \\
& 36 \ta & 156 \ta & 3456 \ta & \\
& 56 \ta & 235 \ta & & \\
& & 236 \ta & & \\
& & 245 \ta & & \\
& & 246 \ta & & \\
& & 346 \ta & & \\
& & 456 \ta & & \\
\hline
\end{array}\]
\centerline{Table 1. Well-roundedness and automorphism groups of lattices from $\ZZ_7$}
\end{figure}

\subsection{Lattices from function fields}

We use the notation of Section~\ref{intro}. In particular, $F$ is an algebraic function field (of a single variable) with the finite field $\FF_q$ as its full field of constants and
 $\calP := \{ P_0,P_1,P_2,\ldots, P_{n-1}\}$ is the set of rational
places of $F$.  The automorphisms  of $F$ permute all places
of a given degree and hence induce  automorphisms of the lattice
$L_\calP$. Thus we may regard $\mathrm{Aut}(F)$ as a subgroup of
$\mathrm{Aut}(L_\calP) \cap S_{n-1}$. Furthermore, each automorphism $\sigma$ of the
divisor class group $\mathrm{Cl}^0(F)$ which permutes the
divisor classes $[P_i -P_0]$ $(1\le i \le n-1)$ also induces an automorphism
of the lattice $L_\calP$. First, note that we may view $\sigma$ as
an element of the symmetric group $S_{n-1}$ by writing $\sigma([P_i - P_0]) = [P_{\sigma(i)} - P_0]$ for
$1\le i \le n-1$.
Second, for every $f  \in O_\calP^*$, $[(f)]$ is the identity
element of $ \mathrm{Cl}^0(F)$ and so, if $(f) = \sum_{i=1}^{n-1} a_i (P_i - P_0)$,
then $\sigma([(f)]) = 0$, that is, the divisor
 $\sum_{i=1}^{n-1} a_i (P_{\sigma(i)} - P_0)$  is again principal, or equivalently,
 $\sum_{i=1}^{n-1} a_i(P_{\sigma(i)} - P_0)$ corresponds to  a lattice point in $L_\calP$.
Let $G$ be the subgroup of $\mathrm{Cl}^0(F)$ which is generated by the divisor classes
$[P_i -P_0]$ $(1\le i \le n-1)$ and let $\mathrm{Aut}(G)^*$ be the group
of all automorphisms of $G$ which permute the divisor classes
$[P_i -P_0]$ $(1\le i \le n-1)$. From Theorem \ref{auto1} it follows that
$$\mathrm{Aut}(\mathrm{Cl}^0(F))^*\cong \Aut(L_\calP) \cap S_{n-1}.$$

\begin{thm}\label{autH}
Let $H$ be a Hermitian function field.
Then the  group  $\mathrm{Aut}(H)$ is isomorphic to a  subgroup of  $\mathrm{Aut}(\mathrm{Cl}^0(H))$.
\end{thm}

\begin{proof}
We write $P_0$ for the place $Q_\infty$. In this proof, the remaining places of $H$ are
$P_1,P_2,  \ldots, P_{n-1}$ where $n = q^3+1$.
Put $G = \mathrm{Cl}^0(H)$.
If $\sigma \in \mathrm{Aut}(H)$, then one can define a map $\phi_\sigma: G \to G$ by
$\phi_\sigma( [\sum_P a_P P] ) = [\sum_P a_P \sigma(P)]$. This map is well-defined:
two degree zero divisor classes  $ D_1 := [\sum_P a_P P], D_2 :=  [\sum_P b_P P]$ are equal
if and only if
$[\sum_P  (a_P - b_P) P]$ is the zero divisor class, that is,  for some function $f$,
$(f) = \sum_P (a_P - b_P) P$.  This is equivalent to
$(\sigma^{-1}(f)) = \sum_P  (a_P - b_P) \sigma(P)$, that is,
to $[\sum_P  (a_P - b_P) \sigma(P)]$ being the zero divisor class.
This is in turn tantamount to saying that $[\sum_P a_P \sigma(P)] = [\sum_P b_P \sigma(P)]$, that is
$\phi_\sigma(D_1) = \phi_\sigma(D_2)$.
It follows from this argument that $\phi_\sigma$ is
well-defined and injective. Since $G$ is finite,  $\phi_\sigma$ is also
surjective. Moreover, $\phi_\sigma$ is a group homomorphism and hence an automorphism of $G$.
Thus we have a  map $\phi:  \mathrm{Aut}(H) \to \mathrm{Aut}(G)$ given by $\sigma \mapsto \phi_\sigma$.
It is quickly checked that $\phi$ is a homomorphism.

Next we show that $\phi$ is injective.
Suppose that $\phi_\sigma$ is trivial for some $\sigma \in \mathrm{Aut}(H)$. Then, for $1\le i \le n$,
  $\phi_\sigma([P_i- P_0]) = [P_i-P_0]$, that is,
 $\sigma(P_i) -P_i+P_0 - \sigma(P_0)$ is principal, and thus the divisor  of a function of degree
 at most $4$. Since $q>2$, according to Theorem~\ref{mindist}, this function must have degree 0. This implies
 that $\sigma(P_i) = P_i$  and $\sigma(P_0) = P_0$, or
 $\sigma(P_i) = P_0$  and $\sigma(P_0) = P_i$  for $1\le i \le n-1$.
 Thus, $\sigma$ is either the identity of $\mathrm{Aut}(F)$ or there is exists an index $j$ $(1\le j \le n-1)$ such
 that $\sigma(P_j) =  P_0$ and $\sigma(P_i) = P_i$ for $i\ne j$ where $1\le i \le n-1$.
 Suppose the latter is the case and that $P_j = P_{\alpha,\beta}$ where $(\alpha,\beta) \in \calK$.
Using Lemma~\ref{tauram},  we obtain that for any $\gamma \in \Fqsq$ such that $\gamma \ne \alpha$, we have
$(x-\gamma)          = \sum_\rho P_{\gamma, \rho}  - q Q_\infty$
and
$( \sigma(x-\gamma) ) = \sum_\rho P_{\gamma, \rho}   - q P_{\alpha,\beta}$
where the sums are over all $\rho \in \Fqsq$ such that $(\gamma,\rho) \in \calK$.
Thus the divisor of $ (x-\gamma)/\sigma(x-\gamma) $ is $P_{\alpha,\beta} - Q_\infty$ and this contradicts
Theorem~\ref{mindist}.  Consequently, $\sigma$ must the be the identity of $\mathrm{Aut}(H)$ and
the map $\phi$ is injective. \end{proof}

\begin{thm}
Let $\mathrm{Aut}(F)^*$ be the group of all automorphisms of $F$ which fix the place $P_0$.
Suppose that $F$ is not the rational function field.
Then the  group  $\Aut(F)^*$ is isomorphic to a  subgroup of  $\mathrm{Aut}(\mathrm{Cl}^0(F))^*$.
\end{thm}

\begin{proof}
Put $G = \mathrm{Cl}^0(F)$ and $A = \mathrm{Aut}(F)^*$.
If $\sigma \in \mathrm{Aut}(F)$, then one can define a map $\phi_\sigma: G \to G$ by
$\phi_\sigma( [\sum_P a_P P] ) = [\sum_P a_P \sigma(P)]$.
As in the proof of Theorem~\ref{autH},  this gives rise to a homomorphism
 $\phi:  A\to \mathrm{Aut}(G)$ via $\sigma \mapsto \phi_\sigma$.
 Next we show that $\phi$ is injective.
Suppose that $\phi_\sigma$ is trivial for some $\sigma \in A$. Then, for $1\le i \le n$,
  $\phi_\sigma([P_i- P_0]) = [P_i-P_0]$, that is,
 $\sigma(P_i) -P_i$ is a principal divisor.
 Since $F$ is not the rational function field,
 $\sigma(P_i) = P_i$  for $1\le i \le n$.
 This implies that there exists a constant $c$ such that  $\sigma(f) = c \cdot f$  for all $f \in F$.
As $\sigma (1)  = 1$, it follows that $\sigma$ is the identity of $A$.

Since $\phi_\sigma([P_i - P_0]) = [\sigma(P_i) - P_0]$ for $1\le i \le n-1$ is a permutation of
the divisor classes $[P_i-P_0]$ $(1\le i \le n-1)$, it follows that
 $\phi$ is in fact a homomorphism from $A$ to  $\mathrm{Aut}(G)^*$.
\end{proof}

 In the case of the Hermitian function field,
the automorphism group $\mathrm{Aut}(H)$ is well understood, see \cite[Page 238]{stich}.
The subgroup $\mathrm{Aut}(H)^*$ has order  $q^3(q^2-1)$ and acts
transitively on the places $P_{\alpha, \beta}$, $(\alpha, \beta) \in \calK$.
Furthermore, the divisor class group of $H$ is $\ZZ_{q+1}^{q^2-q}$.

\section{A lower bound on the number of minimal vectors}
\label{number_min_vecs}

\begin{thm}\label{kiss}
The lattice contains at least $q^7 - q^5 +q^4-q^2$ minimal lattice vectors.
\end{thm}
\begin{proof}
We count the number of functions of the form
$f = L_1/L_2$ where $L_1,L_2$ are lines which satisfy the conditions given in Lemma \ref{minv} for
$(f)$ to be a minimal vector.  We consider each of the cases listed in Lemma \ref{minv}.

{\em Case 1:  $L_1$ and $L_2$ are of the form $x-\alpha$.}
There are $q^2(q^2-1)$ functions of the of this form.

{\em Case 2:  One of $L_1, L_2$  is of the form $x-\alpha$ and the other is a non-tangent line {\rm (}of the form $y + ax + c${\rm )} and both lines have exactly one point of intersection.}
Suppose that $(a,b) \in \calK$ is the point of intersection. Then the lines $L_1 = x-a$ and $L_2 = y-b - m(x-a)$ are two lines of the
required form provided that $m \in \Fqsq$ such that $m\ne a^q$ (by Lemma \ref{tauram}). Thus there are $q^3(q^2-1)$
possibilities for $f$. Since the function $1/f$ gives the lattice vector $-(f)$, we obtain $2q^3(q^2-1)$
minimal lattice vectors in this way.

{\em Case 3:  Both $L_1$ and $L_2$ are non-tangent lines {\rm (}of the form $y+ax+c${\rm )} with a common point of intersection.}
Suppose that $(a,b) \in \calK$ is given. Then the lines $L_1 = y-b - m_1(x-a)$ and $L_2 = y-b - m_2(x-a)$ are two lines of the
required form provided that $m_1,m_2$ are distinct elements of  $\Fqsq$ neither of which is
equal to $a^q$ (by Lemma \ref{tauram}). These are  $q^3(q^2-1)(q^2-2)$ possibilities for the function $f$.

Adding the numbers of minimal vectors obtained from each of the above cases yields the desired result.
\end{proof}

Here is an alternative proof of the above result.
Let $\sigma\in \mathrm{Aut}(H)$.  If $L_1,L_2$ satisfy
the conditions of Lemma \ref{minv} then one can check that $\sigma(L_1/L_2)$ is of the form $c \cdot L_1'/L_2'$ where
$L_1',L_2'$ are again a pair of lines which satisfy one of the conditions of Lemma \ref{minv} and $c$ is a nonzero constant.  Thus,
if we let $T$ be the collection of all functions of the form $c \cdot L_1/L_2$ where  $L_1,L_2$ satisfy
the conditions of Lemma \ref{minv} and $c$ is a nonzero constant, then the group $\mathrm{Aut}(H)$ acts on the set $T$.
 Let $a, b$ be two distinct elements of $\Fqsq$. Then
the function $f := (x-a)/(x-b)$ belongs to $T$.  We show that the orbit of $f$ under the action of
$\mathrm{Aut}(H)$ has $q^2(q^2-1)(q^3+1) = q^7-q^5+q^4-q^2$ elements.
Let $\sigma \in \mathrm{Aut}(H)$. Then  $\sigma(f) = f$
if and only if   $(\sigma(x) - a)/(\sigma(x) - b) = (x-a)/(x-b)$
if and only if   $\sigma(x) =  x$
if and only if $\sigma$ belongs to the Galois group of the extension $H/\Fqsq(x)$, which has order $q$.
Since $\left| \mathrm{Aut}(H) \right| = q^3(q^2-1)(q^3+1)$ (see \cite[Page 238]{stich}),  Theorem~\ref{kiss} follows. A corollary of the above argument is that the group $\Aut(H)$ acts transitively on the set $T$.

\bigskip
\footnotesize{{\sc  Fakult\"at f\"ur Mathematik, TU Chemnitz, 09107 Chemnitz, Germany,
{\em E-mail address:} {\tt aboettch@mathematik.tu-chemnitz.de}}

\smallskip
{\sc Department of Mathematics, Claremont McKenna College, 850 Columbia Ave, Claremont, CA 91711, USA, {\em E-mail address: {\tt lenny@cmc.edu}}}

\smallskip
{\sc Department of Mathematics, Pomona College, 610 N. College Ave, Claremont, CA 91711, USA, {\em E-mail address: {\tt stephan.garcia@pomona.edu}}}

\smallskip
{\sc Department of Mathematics, Claremont McKenna College, 850 Columbia Ave, Claremont, CA 91711, USA, {\em E-mail address: {\tt hmahara@g.clemson.edu}}}}

\end{document}